\documentclass[11pt, oneside]{article}   	
\usepackage{geometry}                		
\usepackage{graphicx}				
\usepackage{amssymb}
\usepackage{amsmath}
\usepackage{amsthm}
\usepackage{harpoon}
\usepackage{accents}

\usepackage{tikz}  
\usepackage{float}
\restylefloat{table}
\usepackage{mathrsfs}
\usepackage{verbatim}
\usepackage{enumitem}  

\usetikzlibrary{positioning,chains,fit,shapes,calc}  

\newtheorem{theorem}{Theorem}
\newtheorem{lemma}{Lemma}
\newtheorem{definition}{Definition}

\newtheorem{remark}{Remark}

\newcommand{\ppp}{\[\begin{aligned}}
\newcommand{\ooo}{\end{aligned}\]}

\newcommand{\iter}{\mathrm{iter}_\phi}
\newcommand{\maxphi}{\mathrm{max}_\phi}
\newcommand{\maxMult}{\mathrm{maxMult}_\phi}

\begin{document}

\title{Multi-point variants of the Newton-Raphson-Simpson method arising from organizing a formal zero according to a function $\phi$ }
\author{Mario DeFranco}
\maketitle

\abstract{Fix an integer $L \geq 1$, and a function $\phi \colon \mathbb{Z}_{\geq 1} \rightarrow [0,L]$ with $\phi^{-1}(\{0\}) = \{ 1\}$. We define a multi-point variant of the Newton-Raphson-Simpson, which we call the max-phi method, as follows. We use $\phi$ to define the iteration number of a rooted plane tree. Then we construct formal series that are weighted generating functions of rooted plane trees with iteration number at most $N$. Finally we use these formulas to define the max-phi method applied to an arbitrary $L$-differentiable function.}

\section{Introduction}

The main result of this paper is the recurrence relation between certain formal series in Theorem \ref{t main}. This recurrence relation naturally suggests the definition of a multi-term variant of the Newton-Raphson-Simpson method,  which we call the max-phi method, in Definition \ref{d maxphi}. See Traub \cite{Traub} and D\u{z}uni\'{c}, Neta, Petkovi\'{c}, and Petkovi\'{c} \cite{Dzunic} for general information multi-point methods. 

In \cite{DeFranco gb}, we defined a formal series $Z$ called a formal zero of a function $f(z)$, depending on a base function $g(z)$ and base zero $\alpha$. In \cite{DeFranco ce}, we considered the formal zero depending on the base function $g(z) = a_0 + a_1z$ with base zero $\alpha=-\frac{a_0}{a_1}$, and indexed the terms of the series by the set of rooted plane trees. This is the formal zero $Z$ defined here in Definition \ref{d RPT} below.

\begin{definition} 
Fix an integer $d \geq 1$. Let $f(z)$ be the function 
\[
f(z) = \sum_{i=0}^d a_i z^i
\]
where $a_i$ are indeterminates. We let $a_i$ denote 0 for $i >d$.
\end{definition}

\begin{definition} \label{d RPT}
A rooted plane tree $T$ is either a single vertex, in which case we denote $T$ by the empty sequence $T=()$; or an ordered sequence $T=(T_1, \ldots, T_l)$ where $T_i$ are rooted plane trees and $l \geq 2$. Let $\mathrm{RootedPlaneTrees}$ denote the set of such $T$. 
Define the weight $w(T)$ by 
\[
w(T) =
 \begin{cases} 
-\frac{a_0}{a_1} &\quad \text{ if } T = () \\ 
-\frac{a_l}{a_1} \prod_{i=1}^l w(T_i) &\quad \text{ if } T = (T_1, \ldots, T_l)
\end{cases}.
\]
Let $Z$ denote the formal series 
\[
Z = \sum_{T \in \mathrm{RootedPlaneTrees}} w(T).
\]
\end{definition}

In \cite{DeFranco NRS}, we described a way of organizing the terms of that formal zero (essentially choosing an order of summation) by assigning to each tree $T$ an iteration number $\mathrm{iter}(T)$ defined next, so that trees of iteration number $N$ appeared in the $N$-th term. 

\begin{definition} \label{d iter}
For $T=()$, define $\mathrm{iter}(T) =0$. For $T = (T_1, \ldots, T_l)$, let $n = \max(\{ \mathrm{iter}(T_i) \colon 1 \leq i \leq l \})$. If there is exactly one such $i$ with $\mathrm{iter}(T_i)=n$, define $\mathrm{iter}(T) =n$. Otherwise define $\mathrm{iter}(T) =n+1$. 
\end{definition} 

We also proved that the partial sums $c_N$ in that particular order of summation satisfied a recursive equation which is identical to that of the Newton-Raphson-Simpson method: 
\[
c_N = c_{N-1} - \frac{f(c_{N-1})}{f'(c_{N-1})}.
\]
(The $c_N$ and $Z$ is well-defined as a formal series in a certain ring $R$ as proved in \cite{DeFranco NRS}, and $f(z)$ is a function $f \colon R \rightarrow R$.)
 
In this paper, we generalize Definition \ref{d iter} to Definition \ref{d iterphi} using an arbitrary function $\phi$. Then, after imposing certain conditions on $\phi$, we obtain the main result (Theorem \ref{t main}) and then define the max-phi method (Definiton \ref{d maxphi}). 
 
\section{Max-phi Methods}

For an integers $a\leq b$, let $[a,b]$ denote the set 
\[
[a,b]=\{ i \in \mathbb{Z} \colon a\leq i \leq b\}.
\]

\begin{definition} \label{d iterphi}
 Let $\phi$ be a function 
\[
\phi \colon \mathbb{Z}_{\geq 1} \rightarrow  \mathbb{Z}_{\geq 0}. 
\]
For $T \in \mathrm{RootedPlaneTrees}$, define the iteration number $\iter(T)$ of $T$ by  
\[
\iter(T) = \begin{cases} 0 &\text{ if } T=()\\ 
   \maxphi(T)+ \phi(\maxMult(T)) &\text{ if } T= (T_1, \ldots, T_l)
\end{cases}.
\]
where
\[
\maxphi(T) = \max(\{\iter(T_i) \colon 1 \leq i \leq l \})
\]
and 
\[
\maxMult(T) = \# \{i \colon  1 \leq i \leq l  \text{ and } \iter(T_i) = \maxphi(T)\}.
\]
\end{definition} 

\begin{definition} 
For an integer $N \geq 0$, let $Z_N$ denote 

\[
Z_N = \sum_{\substack{ T \in \mathrm{RootedPlaneTrees} \\ \iter(T) \leq N }} w(T) \\ 
\]

\end{definition}

\begin{lemma} \label{l krsum}
For $N\geq 0$, we have
\begin{align}
Z_N = & -\frac{a_0}{a_1}+\sum_{k=0}^{\infty} \sum_{r \in \phi^{-1}([0,k])} (Z_{N-k}-Z_{N-k-1})^r \sum_{l=2}^d { l \choose r} (-\frac{a_l}{a_1})Z_{N-k-1}^{l-r}. 
         \end{align}
\end{lemma}

\begin{proof}
Let $T \in \mathrm{RootedPlaneTrees}$ with 
\begin{equation} \label{ineq}
\iter(T) \leq N.
\end{equation}
 The term $ -\frac{a_0}{a_1}$ corresponds to $w(T)$ for $T=()$. Thus assume $T = (T_1, \ldots, T_l)$ with $l \geq 2$ and \[\maxphi(T)=N-k\] and \[\maxMult(T) = r\] for some $k \geq 0$ and $r \geq 1$. Then inequality \eqref{ineq} implies $\phi(r) \leq k$ i.e. $r \in \phi^{-1}([0,k])$. The factor $(Z_{N-k}-Z_{N-k-1})^r$ corresponds the weights of the $r$ trees $T_i$ with $\iter(T_i) = \maxphi(T)$ and the factor $Z_{N-k-1}^{l-r}$ corresponds to the weights of the remaining $l-i$ trees. There are ${ l \choose r}$ ways to order these trees, and $(-\frac{a_l}{a_1})$ completes the weight of $T$. This completes the proof. 
 \end{proof}

\begin{lemma} \label{l krsum L}
Suppose that the image of $\phi$ is $[0,L]$ for some integer $L \geq 1$. Then
\begin{align}
Z_N = &\sum_{k=0}^{L-1} \sum_{r \in \phi^{-1}([0,k])} (Z_{N-k}-Z_{N-k-1})^r \sum_{l=2}^d { l \choose r} (-\frac{a_l}{a_1})Z_{N-k-1}^{l-r}  \label{krsum L sum}\\
         & + (-\frac{a_0}{a_1})+\sum_{i=2}^d  (-\frac{a_i}{a_1}) Z_{N-L}^i. \label{krsum L f}
\end{align}
\end{lemma}
\begin{proof}
The assumption on $\phi$ means for each $k \geq L$ that
\[
\phi^{-1}([0,k]) =  [1,\infty).
\]
From Lemma \ref{l krsum}, it thus suffices to prove  
\[
\sum_{k=L}^{\infty} \sum_{r =1}^\infty { l \choose r} (Z_{N-k}-Z_{N-k-1})^r  Z_{N-k-1}^{l-r} =   Z_{N-L}^l.
\]
By the binomial theorem, the left side is 
\begin{align*}
\sum_{k=L}^{\infty} (Z_{N-k}-Z_{N-k-1}+ Z_{N-k-1})^l  -   Z_{N-k-1}^l   = &\sum_{k=L}^{\infty} Z_{N-k}^l  -   Z_{N-k-1}^l \\ 
=& Z_{N-L}^l.
\end{align*}
 This completes the proof. 
\end{proof}

\begin{theorem} \label{t main} 
Suppose that the image of $\phi$ is $[0,L]$ for some integer $L \geq 1$ and also that $\phi^{-1}(\{ 0\} )=\{ 1\}$. Then
\begin{align}
Z_N = Z_{N-1} - \frac{f(Z_{N-L})+ \sum_{k=1}^{L-1} \sum_{r \in \phi^{-1}([0,k])} (Z_{N-k} - Z_{N-k-1})^r \frac{f^{(r)}(Z_{N-k-1})}{r!}}{f'(Z_{N-1})}. \label{main}
\end{align}
\end{theorem}
\begin{proof}
 The assumption that $\phi^{-1}(\{ 0\}) = \{ 1\}$ means $1 \in \phi^{-1}([0,k])$ for all $ k \geq 0$. Also 
 \[
 \sum_{l=2}^d { l \choose r} (-\frac{a_l}{a_1})x^{l-r} = 
 \begin{cases} 
 -\frac{f^{(r)}(x)}{r!a_1} \text{ if } r \geq 2 \\ 
 1 -\frac{f'(x)}{a_1} \text{ if } r = 1
 \end{cases}.
 \]
Now in the statement of Lemma \ref{l krsum L} we make the following simplifications. Line \eqref{krsum L f} is equal to 
 \begin{align}
&-\frac{f(Z_{N-L})}{a_1}\\ 
&+ Z_{N-L} \label{tel 1}.
 \end{align}
The $k=0$ contribution at line \eqref{krsum L sum} is equal to 
\[
(Z_{N}-Z_{N-1})(1-\frac{f'(Z_{N-1})}{a_1})
\]
because the only possible $r$ is $r=1$ from the assumption $\phi^{-1}(\{ 0\}) = \{ 1\}$. 
This is equal to 
\begin{align}
Z_N(1-\frac{f'(Z_{N-1})}{a_1}) \label{Z_N term}\\ 
+Z_{N-1}\frac{f'(Z_{N-1})}{a_1} \\
 -Z_{N-1}, \label{tel 2}
\end{align}

From the sum over the remaining $k$, we separate the $r=1$ contribution as 
\[
\sum_{k=1}^{L-1}\sum_{r \in \phi^{-1}([1,k])} \left( -(Z_{N-k} - Z_{N-k-1})^r \frac{f^{(r)}(Z_{N-k-1})}{r!a_1} \right)+  (Z_{N-k} - Z_{N-k-1})(1-\frac{f'(Z_{N-k-1})}{a_1})
\] 
which is equal to 

\begin{align}
&\sum_{k=1}^{L-1}  \sum_{r \in \phi^{-1}([0,k])}-(Z_{N-k} - Z_{N-k-1})^r \frac{f^{(r)}(Z_{N-k-1})}{r!a_1} \\
+ &\sum_{k=1}^{L-1 }  (Z_{N-k} - Z_{N-k-1}). \label{tel 3}
\end{align}
The terms from lines \eqref{tel 1}, \eqref{tel 2}, \eqref{tel 3} sum to 0. 
We move line \eqref{Z_N term} to the left side of the equation, so the resulting left side then is equal to 
\[
Z_N \frac{f'(Z_{N-1})}{a_1}.
\]
Then divide by $\frac{f'(Z_{N-1})}{a_1}$. The $a_1$ factors cancel out. This completes the proof. 
\end{proof}

Equation \eqref{main} thus suggests the following definition. 
\begin{definition} \label{d maxphi}
Let $K$ denote a field and $L \geq 1$ an integer. Suppose $\phi$ is a function 
\[
\phi \colon \mathbb{Z}_{\geq 1} \rightarrow   [0,L]
\]
such that $\phi^{-1}(\{ 0\} ) = (\{ 1\})$. Suppose $f$ is a function $f \colon K \rightarrow  K$ such that its first $L$ derivatives exist. Let $(c_0, c_{-1}, \ldots c_{-L+1}) \in K^L$. 
Then for $N \geq 1$ we define the $N$-th iteration $c_N$ by 
\[
c_N = c_{N-1} - \frac{f(c_{N-L})+ \sum_{k=1}^{L-1} \sum_{r \in \phi^{-1}([0,k])} (c_{N-k} - c_{N-k-1})^r \frac{f^{(r)}(c_{N-k-1})}{r!}}{f'(c_{N-1})}
\] 
assuming $f'(c_{N-1} )\neq 0$. We call this method the \emph{max-phi} method with function $\phi$. 
\end{definition}

\begin{remark} 
When $\phi$ further satisfies 
\[
\phi^{-1}([0,k]) = [1,\lambda_k]
\]
for som integers $\lambda_k \leq \lambda_{k+1}$, then the $k$-th term in the sum  
\[
\tilde{f}(c_{N-k}; \lambda_k, c_{N-k-1}) - f(c_{N-k-1})
\] 
where 
\[
\tilde{f}(x; n, c)
\]
is the $n$-th degree Taylor polynomial approximation to $f(x)$ centered at $c$. 

When $L=1$, the only $\phi$ that satisfies the conditions is 
\[
\phi(n) = \begin{cases} 0 \text{ if } n=1 \\ 
1 \text{ if }  n \geq 2
\end{cases} 
\]
and the max-phi method for this $\phi$ is the Newton-Raphson-Simpson method. 
\end{remark} 

\section{Further work}

\begin{itemize} 

\item Obtain explicit formulas for the max-phi iterates $Z_N$ in terms of the zeros of $f(z)$. 

\item Apply the max-phi to other formal zeros that can be expressed as sums over rooted plane trees with negative vertex degree. 

\end{itemize}

\end{document}